\documentclass[reqno,a4paper,12pt]{amsart}
\usepackage{latexsym, amsmath, amssymb, amsthm, a4}
\usepackage{amsfonts, bbold, bbm, marvosym}

\usepackage{mathrsfs}

\usepackage[mathscr]{eucal}
\font\sc=rsfs10 at 12pt

\renewcommand{\a}{\alpha}

\newcommand{\g}{\gamma}
\newcommand{\G}{\Gamma}

\newcommand{\D}{\Delta}
\newcommand{\e}{\epsilon}

\newcommand{\z}{\zeta}

\newcommand{\vt}{\vartheta}
\renewcommand{\i}{\iota}
\renewcommand{\k}{\kappa}

\renewcommand{\l}{\lambda}
\renewcommand{\L}{\Lambda}
\newcommand{\m}{\mu}
\newcommand{\n}{\nu}

\renewcommand{\r}{\rho}

\newcommand{\Si}{\Sigma}

\renewcommand{\t}{\tau}
\newcommand{\f}{\phi}

\newcommand{\F}{\Phi}

\newcommand{\U}{\Upsilon}
\renewcommand{\o}{\omega}
\renewcommand{\O}{\Omega}

\newcommand{\K}{{\mathbb K}}
\newcommand{\R}{{\mathbb R}}

\newcommand{\fip}{\pmb{\f}}

\newcommand{\fb}{{\mathbf f}}
\newcommand{\gb}{{\mathbf g}}
\newcommand{\hb}{{\mathbf h}}

\newcommand{\kb}{{\mathbf k}}

\newcommand{\mb}{{\mathbf m}}

\newcommand{\Hb}{{\mathbf H}}

\newcommand{\Lb}{{\mathbf L}}
\newcommand{\Mb}{{\mathbf M}}
\newcommand{\Nb}{{\mathbf N}}

\newcommand{\Qb}{{\mathbf Q}}

\newcommand{\dF}{\mathfrak d}

\newcommand{\QF}{\mathfrak Q}

\newcommand{\Ac}{{\mathcal A}}
\newcommand{\Bc}{{\mathcal B}}

\newcommand{\Hc}{{\mathcal H}}

\newcommand{\Lc}{{\mathcal L}}

\newcommand{\Nc}{{\mathcal N}}

\newcommand{\Sc}{{\mathcal S}}

\newcommand{\Vs}{\sc\mbox{V}\hspace{1.0pt}}

\newcommand{\supp}{\hbox{{\rm supp}}\,}

\newcommand{\Q}{\QF}

\newcommand{\codim}{\operatorname{codim\,}}

\newcommand{\Tr}{\operatorname{Tr\,}}

\newcommand{\loc}{\operatorname{loc}}

\newcommand{\beq}{\begin{equation}}
\newcommand{\eeq}{\end{equation}}

\numberwithin{equation}{section}
\numberwithin{figure}{section}

\newtheorem{thm}{Theorem}[section]

\newtheorem{lem}[thm]{Lemma}

\theoremstyle{remark}

\theoremstyle{definition}

\newtheorem{exa}[thm]{Example}
\begin{document}

\title[Singular measures]{Lieb-Thirring estimates for singular measures}

\author{Grigori Rozenblum }

\address{Chalmers Univ. of Technology, Sweden; The Euler Intern. Math. Institute, and St.Petersburg State Univ.; Mathematics Center
Sirius Univ. of Sci. and Technol.
Sochi Russia}
\email{grigori@chalmers.se}

\thanks{The author was supported  by the Ministry of Science and Higher Education of the Russian Federation, Agreement  075--15--2019--1619.}

\begin{abstract}
We establish Lieb-Thirring type estimates for the Schr\"odinger operator with a singular measure serving as potential.
\end{abstract}
\maketitle



\section{Introduction}
The Lieb-Thirring (LT) estimates play an important role in the analysis of properties of quantum systems. They concern the sum of powers of moduli of negative eigenvalues of a given self-adjoint operator. Most common are such estimates for a Schr\"odinger operator $\Hb(V)=-\Delta-V$ in $\R^d.$ Denote by $\l_j=\l_j(\Hb(V))$ the negative eigenvalues of $\Hb(V).$ Then the classical \emph{LT estimate} has the form
\begin{equation}\label{1.1}
    \sum_{\l_j<0} |\l_j|^{\g}\equiv \Tr([\Hb(V)]_-)^\g\le \mathbf{LT}(d,\g)\int_{\R^d}V_+(X)^{\frac{d}2+\g}dX.
\end{equation}
This estimate holds with $\g\ge0$ for $d\ge3,$ $\g>0$ for $d=2$, and $\g\ge\frac12$ in dimension $d=1.$ There is an almost 50 years long history of proving this kind of  estimates and of searching best possible constant $\mathbf{LT}(d,\g),$ starting from the initial paper by E.Lieb and W.Thirring \cite{LT}. This history is presented in the recent book \cite{FLWLT} and review papers cited there, and we will not reproduce it here. In this book, several strategies of proving \eqref{1.1} are discussed, leading to different values of $\mathbf{LT}(d,\g)$ and admitting different generalizations.

Probably, the most elementary method of proving \eqref{1.1} consists of deriving this inequality from its special case $\g=0;$ the latter inequality is called \emph{the CLR estimate}, also with almost 50 years' history,
\begin{equation}\label{CLR}
N_-(\Hb(V))\equiv \sum_{\l_j<0}1\le \mathbf{CLR}(d)\int_{\R^d}V_+(X)^{\frac{d}{2}}dX, \, d\ge 3,
\end{equation}
where $N_-(\cdot)$ denotes the number of negative eigenvalues of the operator in question.
 The best constant in \eqref{CLR} is unknown and therefore this method cannot give the optimal constant in \eqref{1.1}.

 More generally, one is interested in LT and CLR type inequalities for the operator,

   \begin{equation*}
    \Hb_l(V)=(-\D)^l-V, \, 0<l<\infty.
   \end{equation*}

 Here, the CLR inequality has the form
 \begin{equation}\label{CLRl}
    N_-(\Hb_l(V))\le C_{\ref{CLRl}}\int V_+(X)^{\frac{d}{2l}}dX;\, 2l<d,
\end{equation}
 while the LT inequality becomes
 \begin{equation}\label{LTl}
  \Tr(\Hb_l(V)_-)^\g\le C_{\ref{LTl}} \int V(X)^{\frac{d}{2l}+\g}dX, \g >0\, \mbox{for}\, d\ge 2l; \g\ge 1-\frac{d}{2l}\, \mbox{for}\, d<2l.
 \end{equation}
 For $l\ne 1$, $d>1$, the sharp constant in \eqref{LTl} is not known yet, see \cite{Foe.Oest}.

 Recently a certain progress was made in the eigenvalue analysis of Schr\"o\-dinger type operators with strongly singular potentials, namely, the ones being  singular measures, see \cite{KarSh1}, \cite{KarSh}, \cite{RSh1}, \cite{RSh2}, \cite{RT3}.  In particular, in dimension $d> 2l$ a version of the CLR estimate was proved, see \cite{RT3}, Corollary 4.4.  We formulate here the particular case of our present interest.
\begin{thm}\label{RTth}
Let $\m$ be a locally finite Borel measure on $\R^d,$ $d>2l$ satisfying the condition
\begin{equation}\label{AR+}
    \m(B(X,r))\le \Ac(\m)r^s,\,s>d-2l,\, 0<r<\infty,\, X\in \Mb\equiv\supp\m,
\end{equation}
where $B(X,r)$ is the ball of radius $r$ centered at $X.$ Suppose that the {density} $V(X)\ge0$ belongs to $L_{\vartheta,\m},$ $\vartheta=\frac{s}{2l-d+s}$, and  consider the measure   $P=V\m.$ Then for the Schr\"odinger operator $\Hb_l(P)=(-\D)^l-P,$ defined by the quadratic form

 \begin{equation}\label{forma}
 \hb_{l,P}[u]=\int |\nabla^l u|^2dX-\int V(X)|u|^2\m(dX),
 \end{equation}
 the estimate holds
\begin{equation}\label{CLRls}
    N_-(\Hb_l(P))\le C(d,l,s)\Ac(\m)^{\vartheta-1}\int V(X)^{\vartheta}\m(dX).
\end{equation}
\end{thm}
We will omit the order $l$ in the notation of operators and quadratic forms further on.

The simple standard calculation used to derive the usual LT inequality from the CLR estimate, does not work directly  for the case of a singular measure. We present it here and show at which point the reasoning breaks down.

 Let $V$ be a function in $L_{\vartheta}(\R^d)$, $d>2l,\g>0$. For $\l>0,$ the function $(V-\l)_+$ belongs to $L_{\frac{d}{2l}}(\R^d)$ and we have, by the variational principle and the usual CLR \eqref{CLRl} applied to the potential $(V-\t)$, the estimate
 \begin{equation}\label{CLRLT}
    N_-(\Hb(V)+\l)=\Nb_-(\Hb(V-\l))\le N_-(\Hb(V-\l)_+)\le C \int((V(X)-\l)_+)^{\frac{d}{2l}}dX.
 \end{equation}
 After this, the substitution of \eqref{CLRLT} into the right-hand side of
 \begin{equation}\label{LTthroughCLR}
    \sum |\l_j(\Hb(V))|^\g =\g\int_0^\infty \l^{\g-1} N_-(\Hb(V)+\l)d\l
 \end{equation}
 leads to the LT estimate \eqref{LTl}.

 Now, if we try to repeat \eqref{CLRLT} with a singular measure $P=V\m$ instead of a function $V$, we see  that $P-\l dX$ is not a singular measure  any more and we may not apply \eqref{CLRls}.  However, this approach can be modified, and this is shown in the present paper. The result is the following.
 \begin{thm}\label{simple LT} For $d>2l$, let $\m$ be a singular measure satisfying \eqref{AR+} with some $s>d-2l$. Then for $\g>0,$
 \begin{equation}\label{LTsing}
    \Tr(\Hb(V\m)_-)^{\g}\le C_{\ref{LTsing}}\Ac(\m)^{\theta-1}\int V_+(X)^{\theta}\m(dX),
 \end{equation}
 with constant $C_{\ref{LTsing}}$ not depending on $V,\m$ and
 \begin{equation}\label{theta}
  \theta\equiv\theta(d,s,l,\g)=\frac{s+2l\g}{s-d+2l}
 \end{equation}
 \end{thm}

For $s=d,$ i.e., for a measure $P$ absolutely continuous with respect to the Lebesgue measure, the exponent $\theta$ equals $\vartheta=\frac{d}{2l}+\g,$ so it  coincides with the exponent in \eqref{LTl}.
Another case which our result can be compared with is the estimate established by R.Frank and A.Laptev, see \cite{FLLT}. There, for $l=1,$ the singular measure $\m$ is the Lebesgue measure on the hyper-plane $\R^{d-1}$ in $\R^d.$ This measure satisfies \eqref{AR+} with $s=d-1$, the exponent $\theta$ in \eqref{theta} equals $d-1+2\g,$ and it coincides with the exponent found in \cite{FLLT}. The reasoning in \cite{FLLT} uses essentially the particular structure of the operator and the separation of variables, therefore the authors, using more specific methods, were able to obtain the sharp value of the constant in the estimate for $d\ge2$ and $\g\ge3/2$ -- which is out of reach for our approach. However, as a special case  of our result, we obtain a generalization of the estimate in \cite{FLLT} with non-sharp constant for any order of the operator, with the hyperplane $x_d=0$ replaced by an arbitrary Lipschitz surface of dimension $s>d-2l$, sufficiently regular at infinity.

The direct approach for proving Theorem \ref{simple LT} covers, naturally, only the set of parameters $(d,l,s)$ for which the CLR estimate is established, namely, $2l<d, s>d-2l.$

 As for the case $2l\ge d,$ we use  a  modification of the direct variational approach, proposed in \cite{Weidl} for the one-dimensional case and later extended to the multi-dimensional one in \cite{NetrWeidl}, \cite{EgKon}. It is described also  in the book \cite{FLWLT} (note that it is based upon the construction present in the original proof of the CLR bound in 1972.). We extend this approach to a wide class of singular measures, so, combined with the basic instruments used in \cite{RT3} in proving the CLR-type estimate, it enables us to establish the proper version of the LT estimate for the whole range of parameters for operators with  a singular measure satisfying \eqref{AR+}. It turns out that this approach works for the case $d>2l$
 as well, so the theorem to follow contains, in particular, an alternative proof of Theorem \ref{simple LT}.
 \begin{thm}\label{Thm.d<2l}Let the measure $\m$ satisfy \eqref{AR+} with $s>d-2l$ for $d>2l,$ alternatively, with $s>0$ for $d\le 2l.$ Let the exponent $\g$ satisfy $\g>0$ for $d\ge 2l$, $\g\ge 1-\frac{d}{2l}$ for $d<2l$.
 Then the estimate \eqref{LTsing} holds for any $V\in L_{\theta,\m},$ $\theta =\frac{s+2l\g}{s-d+2l}.$ \end{thm}

It is interesting to note that in \cite{RT3}, for CLR type estimates, the conditions imposed on the measure $\m$ are different for different relations between $d$ and $2l.$ As it has been already mentioned in Theorem \ref{RTth}, for $d>2l,$ the measure $\m$ must satisfy the upper estimate \eqref{AR+}; for $d<2l,$ an opposite estimate is required, namely $\m(B(X,r))\ge \Bc(\m)r^s,$ while for $d=2l$ an order sharp eigenvalue estimate requires both inequalities for $\m(B(X,r))$, see \cite{RSh2}, \cite{RIntegral}. In the opposite, our LT type inequalities require only the upper estimate \eqref{AR+} for all admissible values of $d,l,s$.

In Section {\ref{prep}} we collect some facts about measures and functional inequalities, needed for further considerations.
For a singular measure $\m$, one should be careful in the definition of the operator $\Hb(V\m)$. This topic is discussed in Section \ref{defin}. Then we present proofs of our main theorems. Finally, we discuss some examples.

The author thanks R.Frank, A.Laptev, and T.Weidl who acquainted him with a preliminary version of their book \cite{FLWLT}.
 \section{Preliminaries}\label{prep}
 \subsection{Geometry considerations} An important fact in measure theory, which our approach is based upon, was established in \cite{RSh2}, Theorem 4.3. The two-dimensional version was proved earlier in \cite{KarSh}, see Lemma 2.13 there.

 We consider only open cubes. For   a fixed cube $\Qb$ in $\R^d$, a cube $Q$ is called \emph{parallel} to $\Qb$ if all one-dimensional edges of $Q$ are parallel to the ones of $\Qb.$
 \begin{lem}\label{LemRS}Let $\m$ be a locally finite Borel  measure on $\R^d$ containing no point masses. Then there exists a cube $\Qb$ such that for any open cube $Q$ parallel to $\Qb,$ measure of the boundary of $Q$ equals zero,
  $\m(\partial Q)=0.$
  \end{lem}
  A simple consequence of Lemma \ref{LemRS}, is the following:
  \begin{lem}\label{LemCov} Let $\m$ be a locally finite Borel measure  on $\R^d$, $\Mb=\supp\m,$ containing no point masses, and $\Qb$ be the cube whose existence is granted by Lemma \ref{LemRS}. Let $X$ be some  point in $\R^d.$ Consider the family of cubes $Q_t(X)$ with edgelength $t$ centered at $X$ and parallel to $\Qb.$ Then for any $\m$-measurable function $F(Y),$ $F\in L_{1,\loc,\m},$ and any $\a>0$ the function $t\mapsto |Q_t(X)|^\a\int_{Q_t(X)}F(Y)\m(dY)$ is continuous for $t\in[0,\infty).$
  \end{lem}

  Lemma \ref{LemRS} together with the Besicovitch covering theorem (see, e.g., \cite{Guz}, Theorem 1.1 or \cite{FLWLT}, Proposition 4.35) leads to the following property.
  \begin{lem} \label{Lem.cov.LT}Let $F$ be a nonnegative function in $L_{1,\m,loc},$ positive on a set of positive measure. In conditions of Lemma \ref{LemRS}, for any   $A>0,\a>0$ it is possible to find a covering $\U$ of $\supp\m$ by cubes $Q$ parallel to each other such that $J(Q,F):=|Q|^\a\int_Q F(Y) \m(dY)=A$ for each cube $Q\in \U$ and the covering can be split into the finite union of (no more than) $\k=\k(d)$ families, $\U=\cup_{j\le \k} \U_j,$ such that in each $\U_j$ the cubes are disjoint. In particular, the multiplicity of the covering $\U$ is not greater than $\k.$
  \end{lem}
  \begin{proof} We fix the cube $\Qb$ given by Lemma \ref{LemRS}. For each $X\in \R^d$, the function $t\mapsto J(Q_t(X),F)$ tends to $+\infty$ as $t\to\infty$ and $J(Q_t(X),F)\to 0$ as $t\to 0.$ By Lemma \ref{LemCov}, $J(Q_t(X),\m)$ is a non-decreasing continuous function of $t$ variable, and therefore there exists a value $t=t(X)$, not necessarily unique, such that $J(Q_{t(X)}(X),F)=A$. Such cubes $Q_{t(X)}(X),$ $X\in\R^d,$ form a covering of $\R^d$, and therefore the existence of a subcovering $\U$ is granted by the Besicovitch theorem.
  \end{proof}
  \subsection{Embedding and trace inequalities}
 Further basic results are the  ones  about the embedding of the Sobolev space into $L_q$-space with respect to a singular measure. Most of them are borrowed from the book \cite{MazBook} or derived from those.

 We suppose that measure $\mu$ satisfies the one-sided estimate \eqref{AR+}
 with some $s>d-2l,$ $s>0.$ As for the exponent $q,$ it is supposed that $q\le\frac{s}{d-2l}$ for $d>2l,$ $q<\infty$ for $d=2l,$ and $q\le\infty$ for $d<2l.$  Such values of $q$ will be called \emph{admissible}. By $\|u\|_{q,\m,Q}$ we denote the norm of a function $u$ in $L_{q,\m}(Q),$ $q\le\infty.$ By $H^l(Q)$ the usual Sobolev space of order $l$ is denoted.

 \begin{lem}\label{maz.embed}For any unit cube $Q_1\in \R^d,$ for an admissible $q$, the inequality holds

 \begin{equation}\label{MazIneq}
\|u\|_{q,\m,Q_1}^2\le C_{\ref{MazIneq}} \Ac(\m)^{\frac2q} \|u\|^2_{H^l(Q_1)},
 \end{equation}
for all $u\in H^l(Q_1)\cap C(\overline{{Q_1}}),$ with constant $C_{\ref{MazIneq}}$ not depending on $u,\m.$
 \end{lem}
 For $2l\le d,$ Lemma \ref{maz.embed}  is a particular case of Theorem 1.4.5 in \cite{MazBook}. For $2l>d,$ \eqref{MazIneq} follows immediately from the embedding of $H^l(Q_1)$ into $C(\overline{Q_1}).$

 Our next point is to find out how the inequality \eqref{MazIneq} changes when the unit cube $Q_1$ is replaced by an arbitrary cube $Q_t$ with edge $t.$

 \begin{lem}\label{Maz.embed.Scaled} For a cube $Q_t\subset \R^d,$ for any $t>0,$ the inequality holds

 \begin{equation}\label{MasIneqScaled}
\|u\|^2_{q,\m,Q_t}\le  C_{\ref{MazIneq}} t^{2l-d+\frac{2s}{q}} \Ac(\m)^{\frac2q} \left( \|\nabla_l u(X)\|^2_{2,Q_t}+ t^{-2l}\|u(X)\|^2_{2,Q_t} \right),
 \end{equation}
 for $u\in H^{l}(Q_t)\cap C(\overline{Q_t})$
 \end{lem}
 \begin{proof} The inequality follows from \eqref{MazIneq} by means of the scaling $X\mapsto tX,$ using the scaling homogeneity properties of the norms involved and the fact that a measure $\m$  transforms under this scaling to the measure $\tilde{\m}$ which satisfies condition of the form \eqref{AR+}, but with $\Ac(\tilde{\m})=t^s\Ac(\m).$
  \end{proof}
  We will also need a trace theorem for functions on the whole space $\R^d,$ both for the case of large dimension, $d>2l$ and low dimension, $d\le 2l.$
  \begin{lem}\label{Lem Embedding} Let  the measure $\m$ in $\R^d$ satisfy \eqref{AR+} with $s>d-2l$, $s>0$. Let $q\in[2,\frac{2s}{d-2l}]$ for $d>2l,$ $q\ge 2$ for $d\le 2l.$ Then for all functions $u\in H^l(\R^d)\cap C(\R^d)$ the inequalities hold
  \begin{equation}\label{Embeddings}
    \|u\|^2_{L_q(\m)}\le C \Ac(\m)^{\frac2q}\|u\|^{2\t}_{L_2(\R^d)}\|\nabla_lu\|_{L_2(\R^d)}^{2-2\t}, \, \t=\frac{d}{2l}-\frac{s}{ql},
    \end{equation}
   \begin{equation}\label{Emb2}
    \|u\|^2_{L_q(\m)}\le C\Ac(\m)^{\frac2q}(\|u\|^2_{L_2(\R^d)}+\|u\|^2_{L_2(\R^d)}),
    \end{equation}
    \begin{equation}\label{emb3}
    \|u\|^2_{L_q(\m)}\le C \Ac(\m)^{\frac2q} t^{d-2l-\frac{2s}{q}}(\|\nabla_lu\|^2_{L_2(\R^d)}+t^{-2l}\|u\|^2_{L_2(\R^d)}).
  \end{equation}
  In particular, for $q=2,$ it follows that the quadratic form $\int|u(X)|^2 \m(dX)$ is infinitesimally bounded with respect to $\|\nabla_l u\|^2_{L_2(\R^d)}$
  \end{lem}

  The first statement \eqref{Embeddings} in Lemma is a particular case of Theorem 1.4.7/1 in \cite{MazBook}; applying here the inequality $a^\t b^{1-\t}\le a \t +b(1-\t),$ we obtain \eqref{Emb2}; finally, using the scaling $X\to t^{\frac12} X,$ we arrive at \eqref{emb3}.

    Suppose now  that a $\m$-measurable function $V\ge0$ belongs to $L_{\theta,\m}(Q_t),$ with $\theta=1$ for $d<2l$, $\theta>1$ for $d= 2l$,  and $\theta\ge \frac{d}{2l}$ for $d>2l$. Then we apply the H\"older inequality and
  Lemma \ref{Maz.embed.Scaled} and obtain the basic estimate.
  \begin{lem}\label{Prop.RN}
  For $u\in H^l(Q_t)\cap C(Q_t),$ $\frac2q+\frac1\theta=1,$
\begin{gather}\label{form.estimate}
    \int_{Q_t} |u(X)|^2V(X)\m(dX)\le \left(\int_{Q_t} V(X)^{\theta}\m(dX)\right)^{\frac1\theta}\left(\int_{Q_t}
    |u(X)|^q\m(dX)\right)^{\frac2q}\le\\\nonumber
    C \Ac(\m)^{\frac{2}{q}} t^{2l-d+\frac{s}{q}}\left(\int_{Q_t} V(X)^{\theta}\right)^{\frac1\theta}   \left(\int_{Q_t} |\nabla u(X)|^2dX+ t^{-2l}\int_{Q_t}|u(X)|^2 dX\right).
  \end{gather}
\end{lem}
 \section{Definition of the Schr\"odinger operator}\label{defin} The Schr\"odinger operator $\Hb({V\m})$ corresponding to the formal differential expression $H_{V\m}=(-\D)^l-V\m$ will be defined by means of the quadratic forms. Here, the complication in the direct definition consists in the fact that for $d\ge 2l$, the Sobolev space $H^l(\R^d)$ is not embedded into the space of continuous functions $C(\R^d).$ Therefore, for a set $\Mb$ of zero Lebesgue measure, the restriction of a function $u\in H^l$ to $\Mb$ is not intrinsically  defined. For $d<2l,$ functions in $H^l(\R^d) $ are continuous and this complication does not arise.

 A detailed study of the restriction of functions in $H^l(\R^d)$  to, possibly fractal, sets of lower Hausdorff  dimension can be found in \cite{Triebel2} and in \cite{Brasche}. We are interested in   more specific results, which admit a more elementary proof.
 \begin{lem}\label{lem.restrict}Let the measure  $\m$ satisfy condition \eqref{AR+} with $s>d-2l$ (i.e., $s>0$ for $2l>d$.) Suppose that $V\ge a_0$ belongs to $L_{\n,\m, \loc}$ where $\n=\frac{s}{s-(d-2l)}$ for $d>2l,$ $\n>1$ for $d=2l$ and $\n=1$ for $d<2l.$ Then the trace operator $\G_C$ from $H^l(\R^d)\cap C(\R^d)$ to $L_{2,V\m,\loc}$ admits a continuous extension $\G:H^l(\R^d)\to L_{2,V\m,\loc}.$\end{lem}
  \begin{proof} By the last statement in Lemma \ref{Embeddings}, we may add an arbitrary constant to the function $V$ and suppose that $V\ge 1.$ To extend $\G$ to the whole of $H^l(\R^d),$ for a given  $u\in H^l(\R^d),$ we take  a sequence $u_n\in H^l\cap C(\R^d) $ converging in $H^l(\R^d)$ to $u$. This sequence is a Cauchy sequence in $H^l(\R^d),$ i.e., $\|u_n-u_m\|_{H^l}\to 0.$ By \eqref{form.estimate}, it follows that on every cube $Q,$ $\G_C(u_n-u_m)\to 0$ in $L_{2,V\m}(Q)$. Thus, $\G_C u_n$ is a Cauchy sequence in $L_{2,V\m}(Q),$ and by the completeness of the latter space, $\G_C u_n$ converges to some $v\in L_{2,V\m}(Q),$ which we accept for the trace of $u$ in $L_{2,V\m}(Q),$ $v=(\G u)|_{Q}.$ Such $v$ should be understood as an equivalence class of functions in $L_{2,V\m}(Q)$, differing on a set  of $V\m$-measure zero. Obviously, such element $v$ does not depend on the choice of the Cauchy sequence $u_n.$ Also, the traces of $u$ corresponding to different intersecting cubes are consistent, as elements in $L_{2,V\m,\loc},$ thus $\G u$ is defined globally.
 \end{proof}

Now we show  that the quadratic form  $\hb_+[u]\equiv\hb_{P_-}[u]=\int_{\R^d}|\nabla_l u^2|dX + \int |u(X)|^2V_-(X)\m(dX),$  defined on $u\in H^l(\R^d)\cap C_0(\R^d)$ for a lower semibounded function $V_-,$ is closable in $L_2(\R^d)$.

 \begin{lem}\label{posit.lemma}Let $\m$ satisfy condition \eqref{AR+} with $s>d-2l$ (i.e., $s>0$ for $2l>d$.) Suppose that $V_-\ge a_0$ belongs to $L_{\n,\m, \loc}$ where $\n=\frac{s}{s-(d-2l)}$ for $d>2l,$ $\n>1$ for $d=2l$ and $\n=1$ for $d<2l.$ Then the quadratic form $\hb_+[u]=\int_{\R^d}|\nabla^lu|^2  dX+ \int V_-(X)|u(X)|^2\m(dX)$ defined on $H^l(\R^d)\cap C(\R^d)$ is closable on $L^2(\R^d)$.
 \end{lem}
\begin{proof} As before, by Lemma \ref{Embeddings}, we can suppose that $a_0>0.$ Let $u_n\in H^l(\R^d)\cap C(\R^d)$ be a Cauchy sequence in $\hb_{P_-}$-metric, $\|u_m-u_n\|^2_{H^l(\R^d)}+\int V_-(X)|u_m(X)-u_n(X)|^2\m(dX)\to 0$, $m,n\to\infty.$ Since the Sobolev space $H^l(\R^d)$ and the weighted space $L_{2,V\m}$ are complete, there exist limits $u_n\to u$ in $H^l(\R^d)$ and $u_n\to v$ in $L_{2,V\m}$. By  estimate \eqref{form.estimate}, for any cube $Q,$ since $V_-\in L_{\n,\m}(Q),$ the restrictions of $u_n$ to $\supp\m\cap Q$ converge to $v$ in $L_{2,V\m}.$ Therefore, $u=v,\,$ $V\m-$almost everywhere. So, if $u=0,$ it follows that $\hb_+[u_n]\to 0, $ and this, by definition, means that the form $\hb_{P_-}$ is closable.
\end{proof}

The  domain of the closure of the form $\hb_{P_-}$ is the set of functions $u\in H^l(\R^d)$ such that $\G u\in L_{2,V\m}.$

Now we add the negative part to the form $\hb_+.$ Let $V=V_+-V_-,$ $V_{\pm}\ge0$,
$V_-\in L_{\n,\m,\loc},$ as in Lemma \ref{posit.lemma}, $V_+\in L_{\n,\m}.$ Then the quadratic form $\hb(V\m)[u]=\int_{\R^d}|\nabla_l u|^2 dX-\int V|u|^2V\m(dX)$ is closed. It will be justified by means of the following
 important inequality for the quadratic form of the Schr\"odinger operator
 \begin{lem}\label{Lem.Neg.part}Suppose that $V_-$   satisfies conditions of Lemma \ref{posit.lemma} and the function $V_+\ge0$ on $\Mb=\supp\m$ satisfies
 $(V_+-a)_+\in L_{\n,\m}$ for some $a\in\R.$ Set $V=V_+-V_-.$ Then the quadratic form
 \begin{equation*}
   \fb[u]=\int |u(X)|^2V_+(X)\m(dX),
 \end{equation*}
 defined on $H^l(\R^d)\cap C(\R^d)$ satisfies
 \begin{equation*}
    \fb[u]\le c_0\int_{\R^d}|\nabla_l u(X)|^2dX+c_1\int_{\R^d}|u(X)|^2 dX
 \end{equation*}
 for some $c_0\in(0,1).$
  \end{lem}
 \begin{proof} By choosing $a$ sufficiently large, we can make $\|(V_+-a)_+\|_{L_{\n,\m}}$ arbitrarily small, less than  a given $\e>0$. So, by \eqref{form.estimate},
 \begin{equation*}
  \fb[u]\le \int (V_+-a)_+|u|^2\m(dX)\le \e \|u\|_{H^l}^2.
 \end{equation*}
 \end{proof}
Thus, by the KLMN theorem, the quadratic form
  
  \begin{equation*}
  \hb[u]=\hb_-[u]-\fb[u]=\int |\nabla^l u(X)|^2 dX-\int V(X)|u(X)|^2 \m(dX) 
  \end{equation*}
  
   is a closable lower semi-bounded form in $L_2$ and it defines a self-adjoint operator $\Hb=\Hb(V\m)$ which we accept for the Schr\"odinger operator $\Hb=(-\Delta)^l -V\m$ in $\R^d.$

  In a similar way we define the \emph{Neumann} operator in a cube $Q$ determined by the  quadratic form
\begin{equation*}
\hb_{V\m,Q}[u]=\int_Q|\nabla_l u(X)|^2dX-\int_QV(X)u(X)\m(dX)
\end{equation*}
defined initially on functions $u\in H^l(Q)\cap C(Q)$ with $V\in L_{\n,\m}(Q).$ We denote this operator by  $\Hb(V\m)_Q^{\Nc}.$

\section{Proofs}\label{SEct.Proofs}
\subsection{The LT estimate. The easy case}\label{LTSect.easy}
We give  the proof of Theorem \ref{simple LT}.
\begin{proof} By the variational principle, it suffices to consider the case  $V\ge0.$ We will use the 'elementary' approach explained in the Introduction. Namely, for a given $\l>0,$ we find a CLR type estimate for the number of negative eigenvalues of the operator $\Hb(V\m)+\l$ and then integrate over $\l\in(0,\infty).$

To do it, we find a lower estimate for the quadratic form of the operator $\Hb(V\m)+\l.$ We have, for $u\in H^l(\R^d)\cap C(\R^d),$
 \begin{gather*}
    ((\Hb(V\m)+\l)u,u)=
    \frac12 \int|\nabla_l u|^2dX \\\nonumber -\left(\int V(X)|u(X)|^2\m(dX)-\frac12 \int_{\R^d}|\nabla_l u|^2 dX-\l\int|u(X)|^2dX\right).
 \end{gather*}
 By estimate in \eqref{emb3}, for $q=2$, setting  $t=c(\l/2)^{-\frac{1}{2l}}$ with proper $c$, 
 \begin{equation}\label{lower.est.2}
 \l\int_{\R^d}|u(X)|^2 dX\ge - \frac12\int_{\R^d}|\nabla_l u|^2dX+C_{\ref{lower.est.2}}\Ac(\m)^{-1}\l^{\frac{s-d+2l}{2l}}\int|u(X)|^2\m(dX).
 \end{equation}

 Therefore, for the quadratic form $\hb_{V\m}+\l$ the lower estimate follows
 \begin{gather}\label{lower.est.3}
  (\hb_{V\m}+\l)[u] \ge \frac12 \int_{\R^d}|\nabla u(X)|^2dX+\\\nonumber C_{\ref{lower.est.2}} \Ac(\m)^{-1}\l^{\frac{s-d+2l}{2l}}\int|u(X)|^2\m(dX)-\int V(X)|u(X)|^2 \m(dX).
 \end{gather}

 It follows that the number of negative eigenvalues of $\Hb(V\m)+\l$ is not greater than the number of such eigenvalues of the  quadratic form on the right-hand side in \eqref{lower.est.3}. To estimate this latter quantity we  apply the CLR bound \eqref{CLRls} and \eqref{LTthroughCLR}:

    \begin{equation*}
    \Nb_-(\Hb(V\m)+\l)\le C \int(V(X)-C\Ac(\m)^{-1}\l^{\frac{s-d+2l}{2l}})_+^{\frac{s}{s-d+2l}}\m(dX).
    \end{equation*}

 Therefore, we arrive at
 \begin{gather}\label{LT.deriv}
    \Tr(\Hb(V\m)_-)^{\g}=\g\int_{0}^\infty\Nb_-(\Hb(V\m)+c\Ac(\m)^{-1}{\l}) \l^{\g-1}d\l\\\nonumber
    \le C\int \m(dX)\int_0^{\infty} (V(X)-C\Ac(\m)^{-1}\l^{\frac{s-d+2l}{2l}})_+^{\frac{s}{s-d+2l}}\l^{\g-1}d\l.
 \end{gather}
 In calculating the integral over $\l$  in \eqref{LT.deriv}, we introduce the new variable $\z=\l (V(X)\Ac(\m))^{\frac{2l}{s-d+2l}}$, and after this change of variables obtain \eqref{LTsing}.
\end{proof}
\subsection{The LT estimate. The hard case.}
\label{LTproof.Sect.Hard}
Now we present the proof of Theorem \ref{Thm.d<2l}. The reasoning covers all values of $d,l,$ therefore, for $d>2l$ this is an alternative proof to the previous one.
As before,
we consider   $V\ge0$. The proof follows the structure of the one in \cite{Weidl}, \cite{NetrWeidl}; see also \cite{FLWLT}.
 By $\mb=\mb(d,l)$ we denote the dimension of the space of polynomials of degree less that $l$ in $\R^d$.

\begin{lem}\label{LemmaN} For some constants $c_0,c_1$ depending on $\g,d,l,s,$ for any cube $Q=Q_t\subset\R^d$ and any $V\ge0,$ $V\in L^{\theta}(Q)$,
\begin{equation}\label{CubeLem0}
    \Nb_-(\Hb(V\m)_Q^{\Nc}+|Q|^{-2l})=0, \, {\mathrm{if}}\, \Ac(\m)^{\frac{2\theta}{q}} |Q|^{\r}\int_Q V^{\theta}\m(dX)\le c_0
\end{equation}
and
\begin{equation}\label{CubeLem1}
\Nb_-(\Hb(V)_Q^{\Nc})\le \mb, \, {\mathrm{if}}\,\,\Ac(\m)^{\frac{2\theta}{q}} |Q|^{\r}\int_Q V^{\theta}\m(dX)  \le c_1,
\end{equation}
where
$\theta=\frac{s+2l\g}{s-d+2l},$ $q=\frac{2\theta}{\theta-1}= \frac{s+2l\g}{\frac{d}{2}-l+l\g},$ $\r=\frac{\theta}{d}[\frac{2s}{q}-d+2l]=\frac{2l\g}{d}.$
\end{lem}

\begin{proof} By definition of $q$,  the inequality \eqref{MasIneqScaled} is valid. Using the H\"older inequality,
we obtain
\begin{gather}\label{long1}
    \int_Q|u(X)|^2V\m(dX)\le \left(\int_Q V(X)^{\theta}\m(dX)\right)^{\frac{1}{\theta}}\left(\int|u(X)|^q\m(dX)\right)^{2/q}\\\nonumber
    \le C_{\ref{long1}}\Ac(\m)^{1-\theta^{-1}}\left(\int_Q V(X)^{\theta}\m(dX)\right)^{\frac{1}{\theta}}|Q|^{1-\frac{s}{d}}\int_Q\left(|\nabla_l u|^2+|Q|^{-\frac{2l}{d}}|u|^2\right)dX.
\end{gather}
Thus, if for $Q=Q_t,$ the coefficient
\begin{equation*}
C_{\ref{long1}}\Ac(\m)^{2/q}\left(\int_Q V(X)^{\theta}\m(dX)\right)^{\frac{1}{\theta}}|Q|^{1-\frac{s}{d}}
\end{equation*}
is not greater than $1,$ or, equivalently,
 $\int_Q V(X)^{\theta}\m(dX)< \Ac(\m)^{1-\theta} |Q|^{\theta(\frac{s}{d}-1)}C_{\ref{long1}}^{-{\theta}} $,
we have
\begin{equation*}
    \int_Q|\nabla_l u|^2dX-\int_Q |u|^2 V\m(dX) \ge -|Q|^{-2\frac{l}{d}}\int_Q|u|^2dX
\end{equation*}
for all $u\in H^l(Q)$. This inequality means that operator $\Hb(V\m)_Q^{\Nc}$ has no spectrum below $-|Q|^{-\frac{2l}{d}}.$
To justify the second assertion of the Lemma, we argue similarly, but apply the Poincar$\mathrm{\acute{e}}$ inequality $|Q|^{-\frac{2l}{d}}\int_Q|u|^2dX\le C\int_Q|\nabla_l u|^2dX$ for functions $u\in H^l(Q)$ subject to $\int_Q u p(X) dX=0$ for all polynomials  $p$ of degree below $l,$ i.e., on the subspace $\tilde{H}^l(Q)$ of functions satisfying the above orthogonality condition.
Therefore, we can repeat \eqref{long1} for functions $u\in\tilde{H}^l(Q)$, omitting the second summand on the right  in the last line in \eqref{long1}, thus obtaining
\begin{equation}\label{ortog.ineq}  \int_Q|u(X)|^2V\m(dX)\le  C_{\ref{ortog.ineq}}\Ac(\m)^{2/q}\left(\int_Q V(X)^{\theta}\m(dX)\right)^{\frac{1}{\theta}} |Q|^{1-\frac{s}{d}}\int_Q(|\nabla_l u|^2)dX.
\end{equation}
Therefore, if $\int_Q V(X)^{\theta}\m(dX)\le \Ac(\m)^{-\frac{2\theta}{q}} C_{\ref{ortog.ineq}}^{-{\theta}}|Q|^{-\theta(1-\frac{s}{d})},$ we have
\begin{equation*}
    \int_Q|\nabla_l u|^2dX-\int_Q |u|^2 V\m(dX) \ge 0, \, u\in \tilde{H}^l(Q).
\end{equation*}
So, the quadratic form of the Schr\"odinger operator $((-\D)^l-V\mu)_Q^{\Nc} $ is nonnegative on a subspace of codimension $\mb(d,l)$ in $H^l(Q),$ and therefore, this operator has not more than $\mb(d,l)$ negative eigenvalues.
\end{proof}
The idea that the covering approach can produce not only eigenvalue estimates but also estimates for LT sums appeared, independently, in the thesis by T.Weidl, see the papers \cite{Weidl} and \cite{NetrWeidl}, and also in \cite{EgKon}.  In fact, it is an improved
  version of the Neumann part of the classical bracketing. We present it for our case, modifying the presentation in \cite{FLWLT}.
\begin{lem}\label{lem.bracketing} Let $\m$ be a locally finite Borel measure, $V\in L_{1,loc,\m},$ $V\ge0$  and let $\U$ be a covering of $\supp V$ by cubes, parallel to each other, such that $\U=\cup_{j\le \k} \U_j,$ and in each $\U_j$ the cubes are disjoint. Then
\begin{enumerate}
  \item for any $\l\ge0,$

  \begin{equation}\label{N.bracketing}
  \Nb_{-}(\Hb(V\m)+\l)\le C\sum_{Q\in\U}\Nb_{-}(\Hb(\k V\m)_Q^{\Nc}+\l), \end{equation}
and
  \item for $\g>0,$
  \begin{equation*}
    \Tr(\Hb(V\m)_-)^{\g}\le C\sum{\Tr((\Hb(\k V\m)_Q^\Nc)_-)^{\g}}.
  \end{equation*}
\end{enumerate}
\end{lem}
 \begin{proof} The second part of Lemma follows from the first one by means of the identity \eqref{LTthroughCLR}. To prove the first statement, we recall the variational principle (the Glazman lemma) in the 'codimension version': for an operator defined by the quadratic form $\gb[u]$ in a Hilbert space $\Hc,$ the number of negative eigenvalues equals the smallest value of codimension of the subspace $\Lc$ in a form-core for $\gb$, such that $\gb[u]\ge0$ for $u\in\Lc.$ The codimension of a subspace is  understood here  as the number of linearly independent continuous linear functionals that are annulled on $\Lc$. Therefore,  to obtain an upper bound for the number of negative eigenvalues, we need to construct a collection of such functionals.  For the operator $\Hb(V\m)+\l,$ the quadratic form is

  \begin{equation*}
  (\hb+\l)[u]=\int|\nabla_l u(X)|^2dX-\int V(X)|u(X)|^2\m(dX)+\l\int|u(X)|^2dX,
 \end{equation*}
and the inequality $(\hb+\l)[u]\ge0$  means
\begin{equation*}
   \int V(X)|u(X)|^2\m(dX)-\l\int|u(X)|^2dX\le\int|\nabla_l u(X)|^2dX
\end{equation*}

 So, let $Q$ be some cube in the covering $\U$ and $w^{Q}_{k}$, be an orthonormal system of eigenfunctions of $\Hb((\k V\m)_Q^{\Nc}+\l)$, corresponding to negative eigenvalues, continued by zero outside $Q$. Each of these functions  generates a functional $\f_{k}^Q$ in $H^l(\R^d)$, the scalar product $\f_k^Q(u)=(u,w_k^Q)$. There are $n_Q=N_-(\Hb(\k V\m)_Q^{\Nc}+\l)$ such functionals, therefore $\Lc_Q,$ the intersection of the null spaces  of these functionals has codimension $n_Q$. On this subspace, the inequality
 \begin{equation}\label{one cube}
    \int_Q V(X)|u(X)|^2\m(dX)\le\l\int_Q|u(X)|^2dX+\int_Q|\nabla_l u(X)|^2dX
 \end{equation}
 holds. 
 
 Now we consider such collections of functionals for all cubes $Q\in\U$  and set $\Lc=\cap_{Q\in\U} \Lc_Q.$ This is the space on which all functionals $\f_k^Q$ annul.
 The subspace $\Lc$ has codimension not greater than the sum of codimensions of all $\Lc_Q,$ $Q\in\U,$
 \begin{equation}\label{codim}
    \codim(\Lc)\le \sum_{Q\in \U}N_-(\Hb(\k V\m)_Q^{\Nc}+\l)).
 \end{equation}
 Now we evaluate the quadratic form $\hb(V\m)$ on $\Lc.$ We sum the inequality \eqref{one cube} over all cubes $Q\in\U:$
\begin{equation}\label{SumOfCubes}
 \sum_{Q\in\U}   \int_Q V(X)|u(X)|^2\m(dX)\le\sum_{Q\in\U}\int_Q \left(\l|u(X)|^2+|\nabla_l u(X)|^2\right)dX.
\end{equation}
For the term on the left in \eqref{SumOfCubes}, since $\U$ is a covering of $\supp V$,
\begin{equation}\label{In.1}
\sum_{Q\in\U}   \int_Q V(X)|u(X)|^2\m(dX)\ge \int_{\R^d}V(X)|u(X)|^2\m(dX).
\end{equation}
On the right in \eqref{SumOfCubes}, since $\U$ is a covering with multiplicity no greater than $\k,$ we have
\begin{equation}\label{In23}
   \sum_{Q\in\U}\int_Q \left(\l|u(X)|^2+|\nabla_l u(X)|^2\right)dX\le \k \int_{\R^d} \left(\l|u(X)|^2+|\nabla_l u(X)|^2\right)dX.
\end{equation}
We substitute \eqref{In.1}, \eqref{In23} in \eqref{SumOfCubes}, which gives
\begin{equation}\label{final sum}
    \int_{\R^d}\k^{-1}V(X)|u(X)|^2\m(dX)\le  \int_{\R^d} \left(\l|u(X)|^2+|\nabla_l u(X)|^2\right)dX.
\end{equation}
The inequality \eqref{final sum} is valid for $u\in \Lc,$ on a subspace of codimension satisfying \eqref{codim}. By the variation principle, this means that
\begin{equation}\label{finalN}
    N_-(\Hb(\k^{-1}V\m)+\l)\le \sum_{Q\in\U}N_-(\Hb(V\m)_Q^{\Nc}+\l),
\end{equation}
which is equivalent to \eqref{N.bracketing}. \end{proof}
Now we finish the proof of Theorem \ref{Thm.d<2l}.
\begin{proof}We follow the reasoning in \cite{FLWLT}. We set $A=\k^{-1}\min(c_0,c_1)$ and apply Lemma \ref{LemCov} with the function $F=V^{\theta}$ and $\a=\r.$ Thus we obtain a covering $\U$ of $\supp F$ by cubes $Q\in\U,$ parallel to each other, such that $|Q|^{\r}\int_{Q}V^{\theta}\m(dX)=A,$ $\U=\cup_{\n=1}^{\k} \U_\i$ and  each family $\U_\i$ consists of disjoint cubes. We denote by $V_j$ the restriction of $V$ to $\Mb\cap Q_j.$
From Lemma \ref{LemmaN} it follows that for each cube $Q\in\U,$ operator $\Hb(V\m)_Q^{\Nc}$ has at most $\mb$ negative eigenvalues, and these eigenvalues, if they exist, are larger than $-C|Q|^{-\frac{2l}{d}},$ therefore,
\begin{equation*}
    \Tr((\Hb(V\m)_Q^{\Nc})_-^{\g})\le C|Q|^{-2l\g/d}=C |Q|^{-\r}.
\end{equation*}
By the choice of $A$, the expression on the left-hand side is not greater than $C \Ac(\m)^{\theta-1}\int_{Q_j}V^{\theta}\m(dX).$ Thus we obtain for each of cubes in $\U:$
\begin{equation}\label{main.1}
    \Tr((\Hb(V\m)_Q^\Nc)_-^{\g})\le C\int_{Q} V^{\theta}(X)\m(dX.)
\end{equation}
Adding these inequalities, by Lemma \ref{lem.bracketing}, we arrive at \eqref{LTsing}.
\end{proof}
\section{Examples}\label{Examples} The leading example of our main Theorem is a measure on a Lipschitz surface in $\R^d.$ Such surface $\Si$, with dimension $m$ and codimension $\dF=d-m$, is locally defined by the equation $y=\fip(x)$ in proper local co-ordinates $X=(x,y)\in\R^m\times \R^\dF$ with Lipschitz $\dF-$ component vector-function $\fip.$ As $\m$ we take the natural surface measure induced by the embedding of $\Si$ into $\R^d$, represented in the above local co-ordinates as $\m(dX)=\left(\det(1+(\nabla\fip(x))^*(\nabla\fip(x))\right)^{\frac12}dx.$  This measure coincides with the $m$-dimensional Hausdorff measure on $\Si.$ We suppose that the Lipschitz constants in all local representations of $\Si$  are bounded by a common quantity $\L.$ In this case, the measure $\m$ satisfies condition \eqref{AR+} with $s=m$ locally, for small $r$, in any neighborhood where the above representation of the surface is valid. We suppose that \eqref{AR+} is satisfied for all $r>0$ with some constant $\Ac.$ This requirement imposes some regularity conditions.
In this case our Theorem \ref{Thm.d<2l} gives the following Lieb-Thirring type estimate.
\begin{exa}\label{CorLip} Let $\Si$ be a Lipschitz surface, as above, and $V(X), X\in\Si$ be a $\m$ measurable function, $V_-\in L_{\theta,\m}(\Si),$ where  $\theta=\theta_{d,m,l,\g}=\frac{m+2l\g}{m+2l-d},$ and $\g$ is a positive number as in Theorem 1.3.
Then
\begin{equation*}
    \Tr(-\D-V\m)_-^\g\le C_{d,l,\g}\Ac(\m)^{\theta-1}\int_{\Si}V(X)^{\theta}\m(dX).
\end{equation*}
\end{exa}
We present some examples of Lipschitz surfaces satisfying the above conditions.
\begin{exa}\label{ex1}\emph{A global Lipschitz graph.} Let $E\subseteq\R^{d-1}$ be a closed set and $y=\f(x), \, x\in E$ be a Lipschitz function on $E$ with Lipschitz constant $\L.$ Then the graph $\Si\in\R^d$ of the function $\f$ satisfies \eqref{AR+} with $\Ac=\o_{m-1}(1+\L^2)^{m/2},$ where $\o_{m-1}$ is the volume of the unit ball in $\R^{m-1},$ and estimate \eqref{LTsing} is valid.
\end{exa}
\begin{exa}\label{Products}\emph{Products.} Let $\Si_1\subset \R^{d_1},$ $\Si_2\subset \R^{d_1}$ be Lipschitz surfaces of dimension $m_\i,$ $\i=1,2$ with Hausforff measures $\m_\i,$  satisfying estimates of the form \eqref{AR+} of order $s_\i=m_\i$ with constants $\Ac_\i$. Consider their direct product
$\Si=\Si_1\times\Si_2\in \R^d,$ $d=d_1+d_2$. One can see that $\Si$ is a Lipshitz surface of dimension $m_1+m_2$ in $R^d$ with constant $\Ac=C\Ac_1\Ac_2.$ The estimate \eqref{LTsing} holds for $\g>0 $ for $2l\le d_1+d_2,$ $\g>1-\frac{d_1+d_2}{2l}$ for $2l>d_1+d_2$. More generally, the estimate \eqref{LTsing} holds for finite products of measures satisfying conditions of the form \eqref{AR+}.
\end{exa}
\begin{exa}\label{Cyl} \emph{Cylinders.} Let $\m_1$ be a measure in $\R^{d_1}$ satisfying a condition of the form \eqref{AR+} with exponent $s_1$. Consider the cylindrical measure $\m=\m_1\otimes \m_2$ in $\R^{d_1+d_2}$ where $\m_1$ is the Lebesgue measure in $\R^{d_2}.$ Such measure satisfies \eqref{AR+} with exponent $s=s_1+d_2.$
\end{exa}

\begin{exa}\label{Fract}{Fractal sets.} We recall the general construction of fractal sets, introduced  by J.Hutchinson, \cite{Hu}. Let $ \pmb{\Sc} = \{\Sc_1, ...\Sc_\kb\}$ be a finite collection  of contractive similitudes (i.e., compositions of a parallel shift, a linear isometry and a contracting homothety)
on $\R^{d}$, $h_1, ..., h_\kb$ are their coefficients of contraction.  We suppose that \emph{the open set condition} is satisfied: there exists an open set $\Vs\subset \R^d$
 such that $\cup \Sc_\i(\Vs)\subset \Vs$ and $\Sc_\i(\Vs)\cap \Sc_{\i'}(\Vs)=\varnothing, \i\ne \i'.$ By the results of Sect. 3.1 (3), 3.2 in \cite{Hu}, there exists a unique compact set $\K=\K(\pmb{\Sc})$ satisfying $\K=\cup_{\i\le \kb} \Sc_j\K.$ This set is, in fact,  the closure of the set of all fixed points of finite compositions of the mappings $\Sc_\i.$ The Hausdorff dimension $s$ of the set $\K(\pmb{\Sc})$ is determined by the equation $\sum h_\i^s=1.$ Let $\m$ be the $s$-dimensional Hausdorff measure $\m_{\pmb{\Sc}}$ on $\K(\pmb{\Sc})$. As explained in \cite{Fra}, Corollary 2.11.(1), p.6696, estimate \eqref{AR+} is valid for such $\m$ with exponent $s.$ Therefore, our result, Theorem \ref{Thm.d<2l}, gives the LT estimate for $\Hb_l(V\m).$
\end{exa}
\begin{exa}\emph{Lipschitz pre-images.} Let $\m$ be a singular measure in $\R^d$ satisfying \eqref{AR+}, $\Mb=\supp\m,$  $\F$ be a Lipschitz mapping of an open set $\O\subset\R^{d'}$ to a neighborhood of $\Mb.$ A measure $\m'$ in $\O$ is induced by this mapping, $\m'(E)=\m(\F(E))$. If $\L$ is the Lipschitz constant for $\F$ then the image of a ball with radius $r$ is inside a ball with radius $\L r,$ therefore, $\m'$ satisfies  condition \eqref{AR+} with the same exponent $s$ and the constant $\Ac(\m')=\L^s\Ac(\m.)$ The general results on LT estimates carry over to the measure $\m'.$
\end{exa}
\begin{exa}{Noncompact  fractals.} Let $\m$ be a fractal measure  in $\R^{d}$ with compact support, as in Example \ref{Fract}, of Hausdorff dimension $s$. Consider a lattice $\Lb$ of rank $m\le s$ in $\R^d$ and construct the measure $\m_\Lb,$ the sum of shifts of $\m$ by vectors in $\Lb$. Such measure $\m_\Lb$ satisfies \eqref{AR+} with the same value of $s.$
\end{exa}
 One can combine constructions in the above examples to obtain more measures for which the LT inequality holds.

\end{document}